\documentclass[reqno]{amsart}
\usepackage{amsmath,amssymb,amsthm}

\usepackage[letterpaper,hmargin=1in,vmargin=1.2in]{geometry}

\numberwithin{equation}{section}

\theoremstyle{plain}

\newtheorem*{theorem*}{Theorem}
\newtheorem{lemma}{Lemma}[section]

\theoremstyle{definition}

\theoremstyle{remark}

\newcommand\real{\mathop{\rm Re}}
\newcommand\imag{\mathop{\rm Im}}
\newcommand*{\defeq}{\mathrel{\vcenter{\baselineskip0.5ex \lineskiplimit0pt
                     \hbox{\scriptsize.}\hbox{\scriptsize.}}}%
                     =}

\author{Jacob Shapiro}
\title{Semiclassical resolvent bounds in dimension two}

\begin{document}
\begin{abstract}
We give an elementary proof of weighted resolvent bounds for
semiclassical Schr\"odinger operators in dimension two. We require the potential function to be Lipschitz with long range decay. The resolvent norm grows exponentially in the
inverse semiclassical parameter, but near infinity it grows linearly. Our result covers the missing case from the work of Datchev.
\end{abstract}
\maketitle 
\author

\section{Introduction}

 Let $\Delta \le 0$ be the Laplacian on $\mathbb{R}^2$. We consider semiclassical Schr\"odinger operators of the form 
\begin{equation}\label{semiclass schro}
P = P_h\defeq -h^2 \Delta + V - E, \qquad E,h > 0.
\end{equation}
Assume that $V \in L^\infty(\mathbb{R}^2)$ is real-valued, and that $ \nabla V$, defined in the sense of distributions, also belongs, to $L^\infty(\mathbb{R}^2).$ The Kato-Rellich Theorem shows that the resolvent $(P - i \varepsilon)^{-1}$ is a bounded linear operator $L^2(\mathbb{R}^2) \to H^2(\mathbb{R}^2)$. We establish bounds on the weighted resolvent.

\begin{theorem*}
Suppose that, for some $\delta_0, c > 0$, the following inequalities hold for almost all $x \in \mathbb{R}^2$,
\begin{equation}\label{Vineq}
V(x) \le c (1 + |x|)^{-\delta_0}, \qquad |\nabla V(x) | \le c (1 + |x|)^{-1 - \delta_0}.
\end{equation}
Then, for any $s>1/2$ there are $C, R, h_0>0$ such that
\begin{equation}\label{theorem1}
\left\| (1+|x|)^{-s} (P -i\varepsilon)^{-1} (1+|x|)^{-s}  \right\|_{L^2(\mathbb R^2) \to L^2(\mathbb R^2)}  \le e^{\frac{C}{h}},
\end{equation}
\begin{equation}\label{theorem2}
\left\|(1+|x|)^{-s}  \mathbf{1}_{\ge R} (P-i\varepsilon)^{-1} \mathbf{1}_{\ge R}(1+|x|)^{-s} \right\|_{L^2(\mathbb R^2) \to L^2(\mathbb R^2)} \le \frac{C}{ h},
\end{equation}
 for all $\varepsilon  >0 $ and $h \in (0,h_0]$, where $\mathbf{1}_{\ge R}$ is the characteristic function of $\{x \in \mathbb R^2 \colon |x| \ge R\}$.
\end{theorem*}
\noindent Burq was the first to prove resolvent bounds of this form in \cite{bu98, bu02}, where he studied Schr\"odinger operators in context of a obstacle problem for the wave equation. He allowed $\Delta$ to be replaced by a Laplace-Beltrami operator with smooth coefficients, and required $V$ to be smooth. Cardoso and Vodev \cite{cavo} extended Burq's result to a wide class of infinite volume Riemannian manifolds. Rodnianski and Tao \cite{rt15} studied Schr\"odinger operators on asymptotically conic manifolds, and obtained resolvent estimates similar to \eqref{theorem1} and \eqref{theorem2}.

In $\mathbb{R}^n$ for $n \ge 3$, Vodev \cite{vod14} studied operators of the same form as in \eqref{semiclass schro}, only $V$ was replaced by $h^\nu V$ for some $\nu >0$, and $V$ was allowed to contain a less regular short range term. Datchev \cite{da} gave an elementary proof of \eqref{theorem1} and \eqref{theorem2} for dimension $n \ge 3$. He only required a decay condition on $\partial_rV$, rather than a decay condition on $\nabla V$. The novel aspect of the Theorem is that \eqref{theorem1} and \eqref{theorem2} are now established in dimension two, when $V$ and $\nabla V$ have low regularity and mild decay.

The $h$ dependence of the resolvent bound in \eqref{theorem1} is well-known to be optimal in general. In particular, in \cite{ddz},  Datchev, Dyatlov, and Zworski established the lower bound 
\begin{equation*}
\left\|(1+|x|)^{-s}  \mathbf{1}_{\ge R} (P-i\varepsilon)^{-1}(1+|x|)^{-s} \right\|_{L^2(\mathbb R^2) \to L^2(\mathbb R^2)} \ge e^\frac{1}{ Ch},
\end{equation*}
for suitable $V$. 
Stronger resolvent bounds are known if $V$ is more regular, and additional assumptions are made about the Hamilton flow $\Phi(t) = \text{exp}t(2 \xi \partial_x - \partial_x V(x) \partial_{\xi}) $. Note that, in our case, $\Phi(t)$ may be undefined, since $\nabla V$ only belongs to $L^\infty( \mathbb{R}^2)$.  For example, if $V$ is  \textit{nontrapping} at the energy $E$, then it is known that \eqref{theorem1} can be improved to
\begin{equation*}
\left\|(1+|x|)^{-s} (P-i\varepsilon)^{-1} (1+|x|)^{-s} \right\|_{L^2(\mathbb R^2) \to L^2(\mathbb R^2)} \le \frac{C}{ h}.
\end{equation*}
For more about resolvent bounds under various dynamical assumptions, see \cite{da}, as well as chapter 6 from \cite{dyzw}, and the references therein. 

Resolvent bounds similar to \eqref{theorem1} and \eqref{theorem2} have proved useful in several applications. Burq used his exponential resolvent bounds in \cite{bu98,bu02} to show logarthimic local energy decay for solutions to the wave equation. As shown in section XIII.7 of \cite{rs78}, the exterior resolvent estimate \eqref{theorem2} is related to exterior smoothing and Strichartz estimates for Schr\"odinger propagators. This is an active research area: see, for instance, \cite{bt07,mmt08} and the papers cited therein. Furthermore, in the recent paper \cite{chr15}, Christiansen used a resolvent bound of the form \eqref{theorem2} to find a lower bound on the resonance counting function on even-dimensional Riemannian manifolds that are flat near infinity and contain a compactly supported perturbation.   

By making $C$ larger and $h_0$ smaller in the Theorem, we can assume without loss of generality that $c = 1/2$. That is, we may assume
\begin{equation}\label{Vineq 1/2}
V(x) \le \frac{1}{2} (1 + |x|)^{-\delta_0}, \qquad |\nabla V(x) | \le \frac{1}{2} (1 + |x|)^{-1 - \delta_0}.
\end{equation}
We may also assume without loss of generality that 
\begin{equation}\label{s delta 0 assumpt}
0 < 2s -1 < \delta_0 < 1/2.
\end{equation}
This is because decreasing $\delta_0$ only weakens the decay on $V$ and $\nabla V$, and increasing $s$ only decreases the weighted resolvent norm. Additionally, to simplify notation, we set $\delta \defeq 2s -1>0$ throughout all of the arguments that follow. 

Our proof hinges on a Carleman estimate similar to those in the papers by Datchev \cite{da} and Cardoso and Vodev \cite{cavo}. The strategy to produce the Carleman estimate for a general dimension $n \ge 3$ is to construct two radial weight functions, $\varphi(r)$ and $w(r)$ that interact favorably with the conjugated operator $r^{-n/2} P r^{n/2}$. This conjugation gives rise to the so-called effective potential term, which takes the form $(n-1)(n-3)(2r)^{-2}$. In dimension $n \ge 3$, the effective potential is positive and decreasing, and can be discarded in the ensuing estimates. But in dimension $n = 2$ only, the effective potential has a negative pole at the origin. The challenge is that $w$ needs to decay sufficiently at the origin to counteract this negative blow-up.  As a result, the Carleman estimate in dimension two comes with a loss at the origin, because $w$ is weak there. But in section  \ref{gluing} we make a resolvent gluing argument that removes the loss and allows us to establish the Theorem. 

The author is grateful to Kiril Datchev for many helpful discussions and suggestions during the writing of this note, and for his support through a research assistantship.

\section{Construction of the weight function}

We use the usual polar coordinates $(r, \theta) \in (0, \infty) \times [0, 2 \pi)$ to denote a point  $(x, y) = (r \cos \theta, r \sin \theta) \in \mathbb{R}^2\setminus\{0\}$. Let $\partial_r V$ denote the radial distributional derivative of $V$. That is,

\begin{equation}\label{partial r V}
\partial_rV(r, \theta) \defeq \partial_{x}V(r, \theta) \cos \theta +  \partial_{y}V(r, \theta) \sin \theta.
\end{equation}
Throughout this section, we need only assume that
\begin{equation}\label{Vrineq}
V \le (1 + r)^{-\delta_0}, \qquad |\partial_rV| \le (1 + r)^{-1 - \delta_0}
\end{equation} 
for almost all $(r, \theta) \in \mathbb{R}^2 \setminus\{0\}$. Note that \eqref{Vineq 1/2} implies \eqref{Vrineq}.

The following two lemmas establish the existence and uniqueness of the radial weight function $\varphi(r)$ that we will use in the Carleman estimate. Lemma \ref{psiconstruct} is due to Datchev \cite[Lemma 2.1]{da}, and it constructs a continuous function $\psi(r)$ that obeys a crucial inequality with $V$, $\partial_rV$, and $E$. Lemma 2.2 is due to Datchev and De Hoop \cite[Proposition 3.1]{ddeh}, and it constructs $\varphi$ as a solution to an ordinary differential equation with right hand side $\psi$.

\begin{lemma}\label{psiconstruct}
For $\delta > 0$ sufficiently small, there exist constants $B, R_0, R_1 >0$ (depending on $\delta$) so that the function

\[
\psi = \psi_\delta(r) \defeq \begin{cases}  \delta_0^{-1} , &r \le R_0, \\  \frac B {1-(1+r)^{-\delta}} - \frac E 4 , & R_0 < r < R_1, \\ 0, &r \ge R_1,\end{cases}
\]
is continuous and satisfies the inequality 
\begin{equation} \label{oldw}-\frac{E}{2} \le \psi - V - (\partial_r V - \psi')\frac{1-(1+r)^{-\delta}}{\delta(1 + r)^{-1-\delta}}
\end{equation}
 for almost all points $(r, \theta) \in \mathbb{R}^2 \setminus \{0\}$.
\end{lemma} 

\begin{lemma}\label{phiconstruct}
 For any $h > 0$, there exists a unique solution  $\varphi = \varphi_h(r) \in C^2([0,\infty))$ to the equation 
\begin{equation} \label{phipsi}
(\varphi'(r))^2 - h \varphi''(r) = \psi(r).
\end{equation}
Furthermore, $\varphi ' \ge 0$, and the support of  $\varphi '$ is contained in $[0,R_0]$ and independent of $h$. 
\end{lemma} 
\noindent Note that, because $\varphi'' = (\psi - (\varphi')^2)/h$, it follows that $\varphi'''(r)$ exists for almost all $r \in [0, \infty)$. 

\section{Proof of the Carleman estimate} \label{Carleman section}

We continue to assume that \eqref{Vrineq} holds throughout this section. Before establishing the Carleman estimate, which is Lemma \ref{Carlemanlemma1}, we need to prove a preliminary inequality. Define the function $w(r)$ to be 

\[
w = w_\delta(r) \defeq \begin{cases} c_0r^2 , &r \le R_0, \\   1 - (1 + r)^{-\delta} , & r >  R_0,\end{cases}
\]
where we set $c_0 \defeq (1 - (1 + R_0)^{-\delta})/R_0^2 $ to make $w$ continuous on $[0, \infty)$. Note that

\begin{equation} \label{w/w'}
\frac{w}{w'} = \frac{1-(1+r)^{-\delta}}{\delta(1 + r)^{-1-\delta}}, \qquad r> R_0.
\end{equation}
Therefore, \eqref{oldw} shows that

\begin{equation} \label{neww}
-\frac{E}{2}w' \le (\psi - V)w' + (-\partial_r V + \psi')w, \qquad r > R_0.
\end{equation}
Set 
\begin{equation*}
V_{\varphi} \defeq V - (\varphi')^2 + h \varphi'' - h^2/4r^2.
\end{equation*}
The inequality we need is as follows.

\begin{lemma}\label{wineq}
 If $\delta>0$ is small enough, then there exists $h_1 > 0$ so that
\begin{equation}\label{E/4}
  \partial_r \left(w(r)(E-V_{\varphi}(r, \theta))\right) \ge \frac{E}{4}w'(r),
\end{equation}
for almost all $(r, \theta) \in \mathbb{R}^2\setminus \{0\}$ and any $h \in (0,h_1]$.
\end{lemma}

\begin{proof}[Proof of Lemma \ref{wineq}]

First, we expand $\partial_r(w(E - V_{\varphi}))$, making use of \eqref{phipsi},


\begin{equation*}
 \partial_r(w(E - V_{\varphi}))= (E -V + \psi)w' + (-\partial_rV + \psi')w + \frac{h^2}{4r^2}\left(w' - \frac{2w}{r}\right). 
\end{equation*}
We will now use this expansion to investigate two cases separately: first when $r \in (0, R_0)$, and then when $r \in (R_0, \infty)$. In the case $r \in (0,R_0)$,  $\psi = \delta_0^{-1}$, and hence $\psi' = 0$. We also have, $w' - 2w/r= 0$. Using these facts, along with the bounds on $V$ and $\partial_rV$ from \eqref{Vrineq}, we arrive at the following inequality for  $\partial_r(w(E - V_{\varphi}))$ when $r \in (0, R_0)$.

  \[\begin{split}
 \partial_r(w(E - V_{\varphi})) &=  \left(E + \delta_0^{-1} - V -\frac{ r\partial_rV}{2}\right)w' 
\\& \ge \left(E + \delta_0^{-1} - \frac{1}{(1+r)^{\delta_0}} -\frac{r}{2(1+ r)^{1 + \delta_0}}\right) w'
\\& \ge \left(E + \delta_0^{-1} - \frac{3}{2}\right)w'
\\& \ge  \frac{E}{4}w'.
\end{split}\]
The last inequality follows because $\delta_0 < 1/2$.

It remains to establish \eqref{E/4} in the case where $r \in (R_0, \infty)$. According to \eqref{neww}, we have

\begin{equation*}
(E - V + \psi)w' + (-\partial_rV + \psi')w \ge \frac{E}{2}w', \qquad r > R_0. 
\end{equation*}
And so, to establish \eqref{E/4} when $r > R_0$, it suffices to show that, for $h$ small enough, we can achieve
\begin{equation} \label{factorw'}
 \frac{h^2}{4r^2}\left(w' - \frac{2w}{r}\right) \ge -\frac{E}{4}w', \qquad r > R_0.
\end{equation}
To this end, define $g(r)$ to be the function

\begin{equation*}
g \defeq g_{\delta}(r)
=  \frac{1}{4r^2} \left(1 - \frac{2}{\delta}\frac{(1 + r)^{1 + \delta} - (1 + r)}{r}\right), \qquad r >0.
\end{equation*}
Observe that $g$ is bounded on the interval $[R_0, \infty)$, and that  for $r \in (R_0, \infty)$, we have
  \begin{equation*}
\left(\frac{ h^2}{4r^2}\right)\left(w' - \frac{2w}{r}\right) = h^2g(r)w'.
\end{equation*}
If we set $h_1 = (E/(4\sup_{[R_0, \infty)}|g|))^{\frac{1}{2}} > 0$, then \eqref{factorw'} holds for $h \in (0, h_1]$.

\end{proof}

 Define $m$ to be the function $ m = m_\delta(r) \defeq (1 + r^2)^{(1 + \delta)/4}$. We now establish the Carleman estimate. 

\begin{lemma} \label{Carlemanlemma1}
Let $\delta$, $h_1$, and $\varphi$ and be as in Lemma \ref{wineq}. Set $h_0 \defeq \min\{1, h_1\}$. There is a $C > 0$ such that 

\begin{equation} \label{Carleman1}
\| (w')^{\frac{1}{2}}e^\frac{\varphi}{h} v \|^2_{L^2(\mathbb{R}^2)} \le \frac{C}{h^2} \| me^{\frac{\varphi}{h}}(P - i\varepsilon) v \|^2_{L^2(\mathbb{R}^2)}   + \frac{C \varepsilon}
{h} \| e^{\frac{\varphi}{h}}v\|^2_{L^2(\mathbb{R}^2)}
\end{equation}
for all $v \in C^\infty_0(\mathbb{R}^2)$, $\varepsilon \ge 0$, and $h \in (0, h_0]$. 
\end{lemma}

\begin{proof}[Proof of Lemma \ref{Carlemanlemma1}]
Let
  \[\begin{split}
P_{\varphi}& \defeq  e^{\frac{\varphi}{h}}r^{\frac{1}{2}}(P - i \varepsilon)r^{-\frac{1}{2}}e^{-\frac{\varphi}{h}}
\\& =  -h^2\partial_r^2 + 2h\varphi' \partial_r + \Lambda + V_{\varphi} -E-i\varepsilon,
\end{split}\]
where
\begin{equation*} 
0 \le \Lambda \defeq -\frac{h^2}{r^2}\Delta_{\mathbb{S}^1},
\end{equation*}
and $\Delta_{\mathbb{S}^1}$ is the spherical Laplacian on the unit circle $\mathbb{S}^1$. 

Next, let $\int_{r, \theta}$ denote the integral over $(0, \infty) \times \mathbb{S}^1$ with respect to $drd\theta$, where $d\theta$ is the usual arclength measure on $\mathbb{S}^1$. Throughout the remainder of the presentation, $C > 0$ will denote a constant depending possibly on $w$, $\varphi, E$, and $\delta$, but not on $u$.  It's precise value will change from line to line, but it will always remain independent of $u$.

To show \eqref{Carleman1} it suffices to prove that 
\begin{equation} \label{intest1}
\int_{r, \theta} \partial_r(w(E- V_{\varphi})) |u|^2 \le \frac{C}{h^2} \int_{r, \theta}\frac{w}{w'} |P_{\varphi}u|^2 + \frac{C\varepsilon}{h} \int_{r, \theta} |u|^2, \qquad u \in e^{\frac{\varphi}{h}}r^{\frac{1}{2}}C^{\infty}_0(\mathbb{R}^2).
\end{equation}  
This is because we can apply \eqref{E/4} along with the fact that $w/w' \le \max\{ 2 / \delta, R_0/2\} m^2$ for all $r \in [0, \infty)$. Additionally, because $w'$ is bounded, we may assume without loss of generality that $\varepsilon \le h$.
Our first step in showing \eqref{intest1} is to define the following functional
\begin{equation*} \label{F}
F(r) \defeq  \| h\partial_r u(r,\theta)\|^2_S - \langle (\Lambda + V_{\varphi}(r,\theta)  - E)u(r,\theta),u(r,\theta)\rangle_S, \qquad r>0.
\end{equation*}
Here, $\| \cdot \|_S$ and $\langle \cdot, \cdot \rangle_S$ denote the norm and inner product on $L^2(\mathbb{S}^1)$, respectively. This functional was used by Cardoso and Vodev \cite{cavo} and by Datchev \cite{da} to prove their own Carleman estimates. 

To condense notation, set $u' \defeq \partial_r u$ and $V_{\varphi}' = \partial_r V_{\varphi}$. We compute the derivative $F$, which exists for almost all $r > 0.$

 \[\begin{split}
F' & =  2 \real \langle h^2u'', u' \rangle_S - 2\real\langle (\Lambda + V_{\varphi} - E)u, u' \rangle_S + 2r^{-1} \langle \Lambda u, u \rangle_S - \langle V_{\varphi}'u, u\rangle_S
\\& = -2\real \langle P_{\varphi} u, u' \rangle_S + 4h\varphi'\| u' \|_S^2 + 2 \varepsilon \imag \langle u , u' \rangle_S + 2r^{-1} \langle \Lambda u, u \rangle - \langle V'_{\varphi} u, u \rangle_S.
\end{split}\]
The calculation of $F'$ is straightforward, but it relies on fact that we can apply the dominated convergence theorem to get,
\begin{equation} \label{diff und int}
\lim_{t \to 0} \int_{\mathbb{S}^1} \frac{V(r + t, \theta)- V(r, \theta)}{t}|u(r, \theta)|^2 d\theta = \int_{\mathbb{S}^1} \partial_r V(r, \theta) |u(r, \theta)|^2 d\theta 
\end{equation}
for almost all $r > 0$. This is a consequence of Fubini's theorem. 
The formula for $F'$ allows us to compute $wF' + w'F$, 
\[\begin{split}
w F' + w' F = &- 2w \real \langle P_{\varphi} u,  u' \rangle_S + \left(4h^{-1} w \varphi' + w' \right)\|h u'\|_S^2  + 2w\varepsilon \imag \langle u, u'\rangle_S \\
& + \left(2wr^{-1} - w'\right) \langle \Lambda u,u\rangle_S + \langle  \left(w(E-V_\varphi)\right)'u,u\rangle_S .
\end{split}\]
If we now use the facts $w \varphi' \ge 0$, $w' > 0$, $\Lambda \ge 0$, $2wr^{-1} - w' \ge 0$, and $-2 \real \langle a,b \rangle + \| b \|^2 \ge -\|a \|^2$, then the preceding inequality implies the following.
\begin{equation} \label{w'F + wF'}
w F' + w' F \ge -  \frac{w^2}{h^2w'} \|P_{\varphi} u\|_S^2 + 2w\varepsilon \imag \langle u, u'\rangle_S +  \langle\left(w(E-V_\varphi)\right)'u,u\rangle_S.
\end{equation}
In addition, Fatou's lemma, along with the fundamental theorem of calculus, show that
\begin{equation} \label{Fatou}
\int_0^\infty (w(r)F(r))'  \le -\liminf_{r \to 0} w(r)F(r) = 0.
\end{equation} 
Integrating \eqref{w'F + wF'} with respect to $dr$ and using \eqref{Fatou}, we arrive at 

\begin{equation} \label{baseeqn}
\int_{r, \theta} (w(E-V_{\varphi}))'|u|^2 \le \frac{1}{h^2} \int_{r, \theta} \frac{w^2}{w'} |P_{\varphi} u|^2 + 2 \varepsilon \int_{r, \theta}w |uu'|.
\end{equation}

We focus on the last term in \eqref{baseeqn}. Our goal is to show 

\begin{equation} \label{wuu'}
2\int_{r, \theta} w |uu'| \le \frac{C}{h} \int_{r, \theta} w^2|P_{\varphi} u|^2 + \frac{C}{h}  \int_{r, \theta} |u|^2. 
\end{equation}
 If we have shown \eqref{wuu'}, we can substitute it into \eqref{baseeqn} to get 

\begin{equation} \label{wuu' 2}
\int_{r, \theta} (w(E - V_{\varphi}))'|u|^2   \le \frac{1}{h^2} \int_{r, \theta} \frac{w^2}{w'} |P_{\varphi} u|^2 + \frac{C\varepsilon}{h} \int_{r, \theta} w^2 |P_{\varphi} u|^2 + \frac{C \varepsilon}{h} \int_{r, \theta} |u|^2,
\end{equation}
If we use the assumptions $\varepsilon \le h$, $h \le 1$, along with the fact $(w^2/w' + w^2) \le (1 + \delta)  w/w'$, we see that \eqref{wuu' 2} implies \eqref{intest1}.

To show \eqref{wuu'}, we first write.

\begin{equation*}
2\int_{r, \theta} w|uu'| \le \frac{1}{h} \int_{r, \theta} |u|^2 + \frac{1}{h} \int_{r, \theta} w^2 |hu'|^2. 
\end{equation*}
We will now show that

\begin{equation} \label{w^2hu'}
\frac{1}{h} \int_{r, \theta} w^2 |hu'|^2 \le \frac{C}{h} \int_{r, \theta} w^2 |P_{\varphi} u|^2 + \frac{C}{h}  \int_{r, \theta} |u|^2, 
\end{equation}
which will complete the proof of the Lemma. To show this, we use integration by parts, along with the facts that $h \le 1$ and $ab \le \gamma a^2/2 + b^2/2 \gamma$ for any $\gamma > 0.$

\[\begin{split} \label{hu'1}
\frac{1}{h} \int_{r, \theta} w^2 |hu'|^2 &= \frac{1}{h} \real \left[ \int_{r, \theta} \bar{u}(-2h^2ww'u') + \int_{r, \theta} \bar{u}(-h^2w^2u'') \right]
\\& \le \frac{(\max w')^2}{\gamma h} \int_{r, \theta} |u|^2 + \frac{\gamma}{h} \int_{r, \theta} w^2 |hu'|^2 + \frac{1}{h} \real \left[ \int_{r, \theta} \bar{u}(-h^2w^2u'') \right].
\end{split}\]
Furthermore, for any $\eta > 0$,
  \[\begin{split} \label{hu'2}
 \frac{1}{h} \real \left[ \int_{r, \theta} \bar{u}(-h^2w^2u'') \right]&= \frac{1}{h} \real \left[ \int_{r, \theta}w^2 \bar{u} (P_{\varphi} - 2h\varphi' \partial_r - \Lambda - V_{\varphi} + E + i \varepsilon )u \right]
\\& \le \frac{1}{h} \int_{r, \theta} w^2|P_{\varphi}u||u| + \frac{2}{h} \int_{r, \theta}w^2 \varphi' |hu'||u| + \frac{1}{h} \int_{r, \theta} w^2 |E- V_{\varphi}||u|^2 
\\& \le  \frac{1}{2h}  \int_{r, \theta} w^2|P_{\varphi}u|^2 + \frac{\eta \max \varphi'}{h} \int_{r, \theta} w^2|hu'|^2 \\&+ \frac{1}{h} \left( \frac{\max (\varphi' w^2)}{\eta} + \max(w^2|E - V_{\varphi}|) + \frac{\max(w^2)}{2} \right)  \int_{r, \theta} |u|^2. 
\end{split}\]
Now, take $\gamma = 1/4, \eta = 1/(4\max \varphi')$. and combine the previous two estimates to get

\begin{equation*}
\frac{1}{h} \int_{r, \theta} w^2 |hu'|^2 \le \frac{1}{2h} \int_{r, \theta} w^2 |P_{\varphi}u|^2 + \frac{C}{h} \int_{r, \theta} |u|^2 + \frac{1}{2h} \int_{r, \theta} w^2 |hu'|^2.
\end{equation*}
If we subtract the last term to the left side of this inequality, and multiply through by 2, we arrive at \eqref{w^2hu'}.

\end{proof} 

\section{Proof of the theorem} \label{gluing}
Set $C_0 = 2 \max \varphi$. Our strategy is to take the Carleman estimate \eqref{Carlemanlemma1} and show that there exist constants $C > 0, R > 0$ so that
\begin{equation} \label{C03}
e^{-\frac{C_0}{h}} \|m^{-1} \mathbf{1}_{\le R} v \|^2 + \|m^{-1} \mathbf{1}_{\ge R} v \|^2 \le \frac{C}{h^2} \|m(P - i \varepsilon)v \|^2 + \frac{C \varepsilon}{h} \|v\|^2,
\end{equation}
for all $v \in C^\infty_0(\mathbb{R}^2)$. Then \eqref{C03} allows us to prove \eqref{theorem1} and \eqref{theorem2} using the same density argument given by Datchev in in \cite{da}, which is independent of dimension. We cannot obtain \eqref{C03} directly from our Carleman estimate \eqref{Carleman1}, because the estimate is weak near the origin. However, the decay assumption on $\nabla V$ from \eqref{Vineq} allows us to make a small shift of coordinates and still maintain \eqref{Vrineq}. We obtain the same Carleman estimate as \eqref{Carleman1} with respect to a new origin. We add the two estimates together and recover \eqref{C03}. This is how we will prove the theorem.   
\begin{lemma}\label{pickx0}
Suppose $V, \nabla V \in L^\infty(\mathbb{R}^2)$ and that $V, \nabla V$ satisfy \eqref{Vineq 1/2}. If
$x_0 \in \mathbb{R}^2$ is chosen so that
\begin{equation}\label{x0}
|x_0| \le 2^{(1+ \delta_0)^{-1}}-1,
\end{equation}
then the functions $V( \cdot - x_0), \partial_rV( \cdot - x_0)$ obey \eqref{Vrineq}.
\end{lemma}
\begin{proof}[Proof of Lemma \ref{pickx0}]
Observe that 
\begin{equation*}
V(x - x_0) \le \frac{1}{2}(1 + |x - x_0|)^{-\delta_0} = \frac{1}{2} \left[ \frac{1 + |x|}{1 + |x - x_0|} \right]^{\delta_0} (1 + |x|)^{-\delta_0},
\end{equation*}
\begin{equation*}
\partial_rV(x - x_0) \le |\nabla V(x - x_0)| \le \frac{1}{2}(1 + |x - x_0|)^{-1 -\delta_0} = \frac{1}{2} \left[ \frac{1 + |x|}{1 + |x - x_0|} \right]^{1 + \delta_0} (1 + |x|)^{-1 -\delta_0}.
\end{equation*}
Because $(1 + |x|)/(1 + |x - x_0|) \le 1 + |x_0|$, it suffices to choose $|x_0|$ small enough so that 
\begin{equation*}
\max \{(1 + |x_0|)^{\delta_0}, (1 + |x_0|)^{1 + \delta_0}\} = (1 + |x_0|)^{1 + \delta_0} \le 2.
\end{equation*}
And this is achieved if we pick $x_0$ to satisfy \eqref{x0}.

\end{proof}

\begin{proof}[Proof of Theorem]
Let $L^2 = L^2(\mathbb{R}^2)$, $H^2 = H^2(\mathbb{R}^2)$, $C_0^\infty(\mathbb{R}^2) = C_0^\infty$. Pick $x_0 \in \mathbb{R}^2$ so that
\begin{equation*}
0 < |x_0| < 2^{(1+ \delta_0)^{-1}}-1.
\end{equation*}
Then $V( \cdot + x_0)$ satisfies the bounds \eqref{Vrineq}, according to Lemma \ref{pickx0}. Therefore, the Carleman estimate \eqref{Carleman1} can be applied to the operator 

\begin{equation*}
P_0 \defeq -h^2 \Delta + V( \cdot + x_0) -E.
\end{equation*}
 We shift coordinates, apply \eqref{Carleman1} with $P_0$ in place of $P$, and then shift back.
\[\begin{split}
\left\| (w'(|\cdot - x_0|))^{\frac{1}{2}} e^{\frac{\varphi(| \cdot - x_0|)}{h}} v \right\|^2_{L^2}   &=\left\| (w')^{\frac{1}{2}} e^{\frac{\varphi}{h}} v ( \cdot + x_0) \right\|^2_{L^2}  \\& \le \frac{C}{h^2} \left\| me^{\frac{\varphi}{h}}(P_0 - i \varepsilon) v( \cdot + x_0) \right\|^2_{L^2} + \frac{C\epsilon}{h} \left\| e^{\frac{\varphi}{h}} v( \cdot + x_0) \right\|^2_{L^2}
\\&=  \frac{C}{h^2} \left\| m(| \cdot - x_0|)e^{\frac{\varphi(| \cdot - x_0|)}{h}}(P - i \varepsilon) v \right\|^2_{L^2}
\\&+ \frac{C\epsilon}{h} \left\| e^{\frac{\varphi(| \cdot - x_0|)}{h}} v \right\|^2_{L^2}
\end{split} \]
Summarizing this estimate in just one line, we have 
\begin{equation} \label{wx0}
\left\| (w'(|\cdot - x_0|))^{\frac{1}{2}} e^{\frac{\varphi(| \cdot - x_0|)}{h}} v \right\|_{L^2} \le  \frac{C}{h^2} \left\| m(| \cdot - x_0|)e^{\frac{\varphi(| \cdot - x_0|)}{h}}(P - i \varepsilon) v \right\|_{L^2} + \frac{C\varepsilon}{h} \left\| e^{\frac{\varphi(| \cdot - x_0|)}{h}} v \right\|_{L^2}.
\end{equation}

To proceed, choose $R >0$ large enough so that $|x| \ge R$ implies that $\varphi(|x|) = \varphi(|x - x_0|) = \max \varphi$. Multiply both \eqref{Carleman1} and \eqref{wx0} through by $e^{-C_0/h}$ to obtain
\begin{equation} \label{C01}
e^{-\frac{C_0}{h}} \|w' \mathbf{1}_{\le R} v \|^2_{L^2} + \|w' \mathbf{1}_{\ge R} v \|^2_{L^2} \le \frac{C}{h^2} \|m(P - i \varepsilon)v \|^2_{L^2} + \frac{C \varepsilon}{h} \|v\|_{L^2}^2, 
\end{equation}
\begin{equation} \label{C02}
e^{-\frac{C_0}{h}} \|w'(| \cdot - x_0|) \mathbf{1}_{\le R} v \|^2_{L^2} + \|w'(| \cdot - x_0 |) \mathbf{1}_{\ge R} v \|_{L^2}^2 \le \frac{C}{h^2} \|m(|x - x_0|)(P - i \varepsilon )v\|_{L^2}^2 + \frac{C \varepsilon}{h} \|v\|_{L^2}^2. 
\end{equation}
Next, note that there exists some constant $K > 0$, depending on $x_0$ and $\delta_0$, so that 
\begin{equation} \label{K}
(m^{-1})^2 \le K ((w')^2 + (w'(|\cdot - x_0|))^2), \qquad  m^2 + (m(|\cdot - x_0|))^2 \le Km^2,
\end{equation}
If we then add \eqref{C01} and \eqref{C02} and apply \eqref{K} to both sides of the inequality, we arrive at
\begin{equation*} 
e^{-\frac{C_0}{h}} \|m^{-1} \mathbf{1}_{\le R} v \|_{L^2}^2 + \|m^{-1} \mathbf{1}_{\ge R} v \|_{L^2}^2 \le \frac{C}{h^2} \|m(P - i \varepsilon)v \|_{L^2}^2 + \frac{C \varepsilon}{h} \|v\|_{L^2}^2,
\end{equation*}
which is \eqref{C03}. 

From this point, we follow reasoning from the proof of the Theorem in \cite{da}. For any $\gamma, \eta >0$, we have 

\[ \begin{split}
2\varepsilon \| v \|^2_{L^2} &= -2 \imag\langle (P - i\varepsilon)v, v \rangle_{L^2} 
\\& \le \gamma^{-1} \| m \mathbf{1}_{\ge R}(P-i \varepsilon)v\|^2_{L^2} + \gamma \| m^{-1} \mathbf{1}_{\ge R}v \|^2_{L^2}\\& + \eta^{-1}\|m \mathbf{1}_{\le R}(P- i \varepsilon)v \|^2_{L^2} 
+ \eta\|m^{-1} \mathbf{1}_{\le R}v\|^2_{L^2}.  
\end{split} \]
Setting $\gamma =h/C$ and $\eta = e^{-2C_0/h} $, we estimate $\varepsilon \| v \|^2_{L^2}$ from above in \eqref{C03} and find that, for $h$ sufficiently small
\begin{equation} \label{C04}
e^{-\frac{C}{h}} \|m^{-1} \mathbf{1}_{\le R} v \|_{L^2}^2 + \|m^{-1} \mathbf{1}_{\ge R} v \|_{L^2}^2 \le e^{\frac{C}{h}} \|m\mathbf{1}_{\le R}(P - i \varepsilon)v \|_{L^2}^2 + \frac{C}{h^2}\|m \mathbf{1}_{\ge R}(P - i \varepsilon)v \|_{L^2}^2.
\end{equation}

The final task is to use \eqref{C04} to show that for any $f \in L^2$,
\begin{equation} \label{prethm}
\begin{split}
e^{-\frac{C}{h}}\|\mathbf{1}_{\le R}(P - i\varepsilon)^{-1} m^{-1} f \|^2_{L^2} &+ \|m^{-1} \mathbf{1}_{R \ge}(P- i\varepsilon)^{-1}m^{-1} f \|^2_{L^2} \\ & \le e^{\frac{C}{h}} \|\mathbf{1}_{\le R}f \|_{L^2}^2 + \frac{C}{h^2} \|\mathbf{1}_{\ge R}f \|^2_{L^2}, 
\end{split}
\end{equation}
from which \eqref{theorem1} and \eqref{theorem2} follow. To establish \eqref{prethm}, we prove a simple estimate and then apply a density argument which relies on \eqref{C04}. 

In what follows, we use $a \lesssim b$ to denote $a \le C_{\varepsilon,h}b$ for $C_{\varepsilon,h}$ depending on $\varepsilon$ and $h$, but not on $v$. By the Kato-Rellich Theorem, $(P - i \varepsilon)^{-1}: L^2 \to H^2$ is bounded. In addition, the commutator $[P,m] = -h^2 \Delta m + 2 h^2 \nabla m \cdot \nabla : H^2 \to L^2$ is bounded. So for $v \in H^2$ such that $mv \in H^2$, we have
\[ \begin{split}
\|m(P - i \varepsilon)v\|_{L^2} & \lesssim \|(P - i \varepsilon)m v \|_{L^2} +  \|[P,m]v \|_{L^2}
\\& \lesssim \| mv \|_{H^2} + \|v \|_{H^2}
\\& \lesssim  \| mv \|_{H^2}.
\end{split} \]
Thus we have shown 
\begin{equation} \label{Ceph}
\|m(P-i\varepsilon)v\|_{L^2} \le C_{\varepsilon,h} \|mv\|_{H^2}, \qquad \text{$v \in H^2$ such that $mv \in H^2$}. 
\end{equation}

For fixed $f \in L^2$, the function $m(P-i\varepsilon)^{-1}m^{-1} f \in H^2$ because 
\[ \begin{split}
m(P-i\varepsilon)^{-1}m^{-1} f &= (P - i\varepsilon)^{-1} f + [m, (P-i\varepsilon)^{-1}] m^{-1}f  
\\& =  (P - i\varepsilon)^{-1} f + (P -i\varepsilon)^{-1} [P,m]  (P -i\varepsilon)^{-1} m^{-1}f.
\end{split} \]
Now, choose a sequence $v_k \in C_0^\infty$ such that $ v_k \to  m(P-i\varepsilon)^{-1}m^{-1} f$ in $H^2$. Define $\tilde{v}_k \defeq m^{-1}v_k$. Then, as $k \to \infty$
\begin{equation*}
\| m^{-1} \tilde{v}_k - m^{-1} (P- i \varepsilon)^{-1}m^{-1}f \|_{L^2} \le \| v_k - m (P- i \varepsilon)^{-1}m^{-1}f \|_{H^2} \to 0.
\end{equation*}
Also, applying \eqref{Ceph}
\begin{equation*}
\|m(P- i \varepsilon)\tilde v_k - f\|_{L^2} \lesssim \|v_k - m (P- i \varepsilon)^{-1} m^{-1} f \|_{H^2} \to 0.
\end{equation*} 
We then achieve \eqref{prethm} by replacing $v$ by $\tilde{v_k}$ in \eqref{C04} and sending $k \to \infty$.
\end{proof}

\end{document}